\documentclass[reqno, 12pt]{amsart}
\usepackage[mathscr]{eucal}
\usepackage{amscd}
\usepackage{color}
\usepackage{amsfonts}
\usepackage{amsmath, amsthm, amssymb}
\usepackage{latexsym}
\usepackage{psfrag}

\usepackage{graphicx}

\numberwithin{equation}{section}
\newcommand{\ep}{\epsilon}
\newcommand{\R}{\mathbb{R}}

\newcommand{\de}{\delta}

\newcommand{\eps}{\ep}
\newcommand{\ga}{\gamma}

\newcommand{\Om}{\Omega}
\newcommand{\la}{\lambda}
\newcommand{\ol}{\overline}

\newcommand{\pl}{\partial}
\newcommand{\vep}{\varepsilon}

\newcommand{\al}{\alpha}
\newcommand{\avint}{\rlap{$-$}\!\int_{\Omega}}

\DeclareMathOperator*{\esssup}{ess\,sup}

\newtheorem{theorem}{Theorem}[section]
\newtheorem{lemma}{Lemma}[section]
\newtheorem{proposition}{Proposition}[section]

\newtheorem{rem}{Remark}[section]

\def\be{\begin{equation}}
\def\ee{\end{equation}}
\def\bge{\begin{eqnarray}}
\def\bgee{\begin{eqnarray*}}
\def\ege{\end{eqnarray}}
\def\egee{\end{eqnarray*}}
\begin{document}

\title{On the dynamics of a non-local parabolic equation arising from the Gierer-Meinhardt system}

\author{Nikos I. Kavallaris}
\address{
 Department of Mathematics, University of Chester, Thornton Science Park
Pool Lane, Ince, Chester  CH2 4NU, UK
}

\email{n.kavallaris@chester.ac.uk}

\author{Takashi Suzuki}

\address{Division of Mathematical Science\\
Department of System Innovation\\
Graduate School of Engineering Science\\
Osaka University\\
Machikaneyamacho 1-3\\
Toyonakashi, 560-8531, Japan}

\email{suzuki@sigmath.es.osaka-u.ac.jp}

\subjclass{Primary: 35B44, 35K51 ; Secondary: 35B36, 92Bxx }

\keywords{Pattern formation, Turing instability, activator-inhibitor system, shadow-system, invariant regions, diffusion-driven blow-up}

\date{\today}
\maketitle
\begin{abstract}
The purpose  of the current paper is to contribute to the comprehension of the dynamics of the shadow system of an activator-inhibitor system known as a Gierer-Meinhardt model. Shadow systems are intended to work as an intermediate step between single equations and reaction-diffusion systems. In the case where the inhibitor's response to the activator's growth is rather weak, then the shadow system of the Gierer-Meinhardt model is reduced to a single though non-local equation whose dynamics will be investigated. We mainly focus on the derivation of blow-up results for this non-local equation which can be seen as instability patterns of the shadow system. In particular, a {\it diffusion driven  instability (DDI)}, or {\it Turing instability},  in the neighbourhood of a constant stationary solution, which it is destabilised  via diffusion-driven blow-up, is obtained. The latter actually indicates the formation of some unstable patterns, whilst some stability results of global-in-time solutions towards non-constant steady states guarantee the occurrence of some stable patterns.

\end{abstract}

\section{Introduction}
In as early as 1952, A. Turing in his seminal paper \cite{t52} attempted, by using reaction-diffusion systems, to model the phenomenon of {\it morphogenesis}, the regeneration of tissue structures in hydra, an animal of a few millimeters in length made up of approximately 100,000 cells. Further observations on the morphogenesis in hydra led to the assumption of the existence of two chemical substances (morphogens), a slowly diffusing (short-range) activator and a rapidly diffusing (long-range) inhibitor. A. Turing was the first to indicate that although diffusion has a smoothing and trivializing effect on a single chemical, for the case of the interaction of two or more chemicals different diffusion rates could force the uniform steady states of the corresponding reaction-diffusion systems to become unstable and to lead to nonhomogeneous distributions of such reactants. Such a phenomenon is now known as {\it diffusion driven  instability (DDI)}, or {\it Turing instability}.

Exploring Turing's idea further,  A. Gierer and H. Meinhardt, \cite{gm72}, proposed in 1972 the following activator-inhibitor system, known since then as a Gierer-Meinhardt system, to model the regeneration phenomenon of hydra located in a domain $\Om\subset \R^N, N\geq 1$ \bge
&& u_t= \ep^2 \Delta u-u+\displaystyle\frac{u^p}{v^q}, \quad\mbox{in}\quad \Om \times (0,T), \label{gm1}\\
&& \tau v_t = D \Delta v-v+\displaystyle\frac{u^r}{v^s}, \quad\mbox{in}\quad \Om \times (0,T),\label{gm2}\\
&&\displaystyle\frac{\partial u}{\partial \nu}=\displaystyle\frac{\partial v}{\partial \nu}=0,\quad\mbox{on}\quad \pl\Om\times (0,T),\label{gm3}\\
&& u(x,0)=u_0(x)>0,\quad v(x,0)=v_0(x)>0,\quad\mbox{in}\quad \Om, \label{gm4}
\ege
where $\nu$ denotes the unit outer normal vector to $\pl \Om$ whilst $u$ and $v$ stand for the concentrations of the activator and the inhibitor respectively. System \eqref{gm1}-\eqref{gm4} intends to provide a thorough explanation of symmetry breaking as well as of {\it de novo} pattern formation by virtue of the coupling of a local activation and a long-range inhibition process. The inserted nonlinearities describe the fact that the activator promotes the differentiation process and it stimulates its own production, whereas  the inhibitor acts a suppressant against the self-enhancing activator to prevent the unlimited growth.

Here, $\ep^2, D$ represent the diffusing coefficients whereas the exponents satisfying the conditions:
$$p>1,\; q, r,>0,\quad\mbox{and}\quad s>-1,$$ measure the morphogens interactions. In particular, the dynamics of system \eqref{gm1}-\eqref{gm4} can be characterised by two numbers: the {\it net self-activation index} $\rho\equiv (p-1)/r$ and the {\it net cross-inhibition index} $\gamma\equiv q/(s+1).$ Indeed, $\rho$ correlates the strength of self-activation of the activator with the cross-activation of the inhibitor. So, if $\rho$ is large, then the net growth of the activator is large no matter the inhibitor's growth. On the other hand, $\gamma$ measures how strongly the inhibitor suppresses the production of the activator and that of itself. Now if $\gamma$ is large then the production of the activator is strongly suppressed by the inhibitor. Finally, the parameter $\tau$ quantifies the inhibitor's response against the activator's growth.

Guided by biological interpretation as well as by mathematical reasons, it is usually assumed that the parameters $p,q,r,s$ satisfy the following condition
\bgee
\rho\equiv\frac{p-1}{r}<\displaystyle\frac{q}{s+1}\equiv \gamma,
\egee
or equivalently
\begin{equation}\label{turing2}
p-r\gamma <1
\end{equation}
Condition \eqref{turing2} is called a {\it Turing condition} whilst
the reverse inequality
\begin{equation}\label{aturing}
p-r\gamma > 1
\end{equation}
will be referred to as an {\it anti-Turing condition}.

The {\it Turing condition} guarantees, \cite{nst06}, that the spatially homogeneous equilibrium $(u,v)=(1,1)$ of the corresponding kinetic (ODE) system
\begin{equation}
\displaystyle\frac{du}{dt}=-u+\displaystyle\frac{u^p}{v^q}, \quad
\tau \displaystyle\frac{dv}{dt}=-v+\displaystyle\frac{u^r}{v^s}.
 \label{ode1}
\end{equation}
is stable if $\tau<\frac{s+1}{p-1}.$ Nevertheless, once diffusion terms are introduced, with $\ep^2\ll D,$ and under \eqref{turing2} then $(u,v)=(1,1)$ becomes unstable and bifurcation occurs, see also \cite{nst06}. Therefore, {\it diffusion driven instability (DDI)} takes place which leads to pattern formation and then explains the phenomenon of morphogenesis.

Apart from its vital biological importance the system \eqref{gm1}-\eqref{gm4} has also interesting mathematical features and emerging singularities. As such, it has attracted a lot of attention from the field of mathematical analysis. Subjects of interest include the existence of global-in-time solutions, which was first investigated in \cite{r84} and then studied more thoroughly in \cite{msy95, mt87}. The author in \cite{j06} proved that under the condition
$\frac{p-1}{r}<1$, a global-in-time solution exists, which is an almost optimal result, also taking into consideration the results in \cite{nst06}. Furthermore, \cite{ksy13} contains an investigation of  the asymptotic behaviour of the solution of (\ref{gm1})-(\ref{gm4}). In particular the authors showed that if $\tau=\frac{s+1}{p-1}$, $s>0$, and
\bgee\frac{2\sqrt{d_1\,d_2}}{d_1+d_2}\geq \sqrt{\frac{(s+1)(p-1)}{sp}}, \quad d_1=\ep^2, \ d_2=\tau^{-1} D, \egee
then the global-in-time solution of (\ref{gm1})-(\ref{gm4}) is approaching uniformly a spatially homogeneous solution, which is always periodic-in-time unless it is a constant one. The occurrence of finite-time blow-up, which actually means unlimited growth for the activator,  was first established in \cite{msy95} and later in \cite{ksz16, lps17, z15}, whereas the case of nondiffusing activator finite-time blow-up is also investigated in \cite{ksz16}. The existence and stability of spiky stationary solutions is thoroughly studied in the survey paper \cite{wei08}.

As specified above, in the case of the Gierer-Meinhardt system, the inhibitor diffuses much faster compared to the activator, i.e. $\eps^2\ll D,$ and thus the system (\ref{gm1})-(\ref{gm4}) can be fairly approximated by its {\it shadow system} when $D\gg 1.$ The concept of a {\it shadow system} was introduced by Keener, \cite{ke78}, to describe the qualitative behaviour of reaction-diffusion systems when one of the diffusing coefficients is very large. Such a system is formed by a reaction-diffusion equation coupled with an ordinary differential equation (ODE) with non-local effects and it actually contains all the essential dynamics of the original reaction-diffusion system. In particular, if there is a compact attractor for the  {\it shadow system} the original reaction-diffusion system has a compact attractor too, see also \cite{hs89}.

In the following we provide a formal derivation of the {\it shadow system} of the Gierer-Meinhardt system (\ref{gm1})-(\ref{gm4}). A rigorous proof can be found in \cite{mm17, mchks} where it is also shown that the convergence  of the original reaction-diffusion system towards its {\it shadow system} is valid locally in time except for an initial layer. Now, dividing (\ref{gm2}) by $D$ and letting $D\uparrow+\infty$ for any fixed $t\in(0,T),$ then due to the boundary condition \eqref{gm3} $v$  becomes spatial homogeneous, i.e. $v(x,t)=\xi(t).$  Next, integrating the resulting equation over $\Om$  we finally derive that $u(x,t),\xi(t)$ satisfy the {\it shadow system}:
\begin{eqnarray}
& & u_t=\ep^2 \Delta u-u+\displaystyle\frac{u^p}{\xi^q}, \quad\mbox{in}\quad \Om \times (0,T), \label{sd1}\\
& & \tau \xi_t=-\xi+\displaystyle\frac{1}{\xi^s} \displaystyle\avint u^r \,dx \quad\mbox{in}\quad \Om \times (0,T),\label{sd2}\\
& & \displaystyle\frac{\partial u}{\partial \nu}=0, \quad\mbox{on}\quad \partial \Om\times (0,T),\label{sd3}\\
& & u(x,0)=u_0(x)>0,\quad  \xi(0)=\bar{v}_0=\frac{1}{|\Om|}\int_{\Om} v_0,\quad\mbox{in}\quad \Om,
 \label{sd4}
\end{eqnarray}
where
\bgee
\avint u^r \, dx\equiv\frac{1}{|\Om|} \int_{\Om} u^r \, dx.
\egee
Note that \eqref{sd1}-\eqref{sd4} is  non-local due to the presence of the integral term in (\ref{sd2}).

Since the convergence towards \eqref{sd1}-\eqref{sd4} holds only locally in time there might be discrepancies in global-in-time dynamics between of the (\ref{gm1})-(\ref{gm4}) and those of \eqref{sd1}-\eqref{sd4} for some range of the involved parameters $p,q,r,s;$  this has been also indicated in \cite{ln09, ly14}.  On the other hand, there are ranges of the involved parameters, where the two systems have exactly the same long-time behaviour \cite[Theorem 1]{ln09} and thus it is worth investigating the {\it shadow system} \eqref{sd1}-\eqref{sd4}, which is simpler compared to the full system (\ref{gm1})-(\ref{gm4}), so we can capture some of the features of (\ref{gm1})-(\ref{gm4}).


Henceforth, we focus on the case where $\tau=0;$ i.e. when the inhibitor's response rate is quite small against the inhibitor's growth. We will investigate the dynamics of \eqref{sd1}-\eqref{sd4} in a forthcoming paper.  For $\tau=0$ the second equation  \eqref{sd2} is solved as
\[ \xi(t)=\left( \avint u^r(x,t)\,dx\right)^{\frac{1}{s+1}}, \]
and thus the {\it shadow system} reduces to the following non-local problem
\bge
&& u_t= \Delta u -u +\displaystyle\frac{u^p}{\Big(\displaystyle\avint u^r \, dx\Big)^{\gamma}}, \quad \mbox{in}\quad \Omega \times (0, T),\label{bsd1}\\
&& \displaystyle\frac{\partial u}{\partial \nu}=0, \quad\mbox{on}\quad \partial \Om\times (0,T),\label{bsd2}\\
&& u(x,0)=u_0(x)>0,\quad  \mbox{in}\quad \Om,
 \label{bsd3}
\ege
where for simplicity it has been considered $\ep=1.$

The rest of the current work is devoted to the study of problem the \eqref{bsd1}-\eqref{bsd3} whose mathematical structure is intriguing. In particular, due to the presence of the non-local term and the monotonicity of its nonlinearity, then problem  \eqref{bsd1}-\eqref{bsd3} does not admit a maximum principle, \cite{qs07}, and so alternatives to comparison techniques should be employed to investigate its long-time behaviour. Some global-in-time existence and blow-up results for problem \eqref{bsd1}-\eqref{bsd3} were presented in \cite{ly14}, whereas some slow moving spike solutions were constructed in \cite{iw00}.  In the current paper, we provide novel global-in-time and blow-up results, extending further the mathematical analysis provided in \cite{ly14}, as well as describing the form of the destabilization patterns developed due to the phenomenon of {\it DDI}.

In addition, the investigation of the non-local problem \eqref{bsd1}-\eqref{bsd3} is also attractive from the biological point of view. Specifically, it will reveal under which circumstances the dynamics of the interaction of the two morphogens (activator and inhibitor) can be controlled by governing only the dynamics of the activator itself.

The rest of the manuscript is composed of nine sections. In the next section we provide the main notation  as well as some preliminary results used throughout the manuscript. Our main results, which include the existence of global-in-time and  blowing up solutions of \eqref{bsd1}-\eqref{bsd3}, are presented in section \ref{sec:mr}. Section \ref{sec:qi} contains the proof of a lower estimate of any solution of \eqref{bsd1}-\eqref{bsd3} which actually guarantees its well posedness.  In section \ref{sec:2} we treat the special case $r=p+1,$ whence problem \eqref{bsd1}-\eqref{bsd3} has a variational structure, and under the Turing condition we  prove that global-in-time solutions converge towards steady states through Turing patterns. Section \ref{sec:3} contains a global-in-time existence result for \eqref{bsd1}-\eqref{bsd3} analogous to the one in \cite{j06}. In section \ref{sec:4} we derive proper estimates of $L^{\ell}-$norms, $\ell>1,$ of the solution $u(x,t)$ which either lead to global-in-time existence or to finite-time blow-up. A {\it DDI} result, which is actually exhibited in the form of a diffusion-driven blow-up for peaky initial data,  is proven in section \ref{sec:5}. Finally section \ref{sec:7} investigates the blow-up rate as well the blow-up profile of the derived blowing up solutions and section \ref{conc} summarizes the main conclusions of the current work.

\section{Preliminaries}\label{prl}
Throughout the manuscript by $||h||_{\ell}, \;\ell>0,$ we denote the $L^{\ell}-$norm of function $h$ defined as:
\bgee
||h||_{\ell}:=\left(\avint |h|^{\ell}\,dx\right)^{1/\ell}, \quad 0<\ell<\infty,
\egee
and
\bgee
||h||_{\infty}:=\esssup_{\Om} |h|,
\egee
whereas $||h||_{H^1}$ stands for the norm of the Sobolev space $H^1(\Om)$ defined as:
\bgee
||h||_{H^1}:=\left(\avint (|h|^{2}+|\nabla h|^2)\,dx\right)^{1/2}.
\egee
Moreover, if $\Delta$ is the Laplace operator associated with Neumann boundary conditions then by $e^{t\Delta}$ we denote its semigroup. Then the well-known estimate, \cite{r84}, holds
\begin{equation}
\Vert e^{t\Delta}h\Vert_q\leq C\max\{ 1,
t^{-\frac{N}{2}(\frac{1}{\ell}-\frac{1}{q})}\} \Vert h\Vert_\ell, \quad 1\leq \ell\leq q\leq \infty.
 \label{eqn:1.2}
\end{equation}

Note that under condition \eqref{turing2}  the solution of the spatially homogeneous part
\bge\label{ts1}
\frac{du}{dt}=-u+u^{p-r\gamma}, \quad u(0)=\bar{u}_0>0,
\ege
never exhibits blow-up, since the non-linearity is sublinear, and its unique stationary state $u=1$ is asymptotically stable. Below, by using stability analysis we show that condition \eqref{turing2} implies linear instability.

Indeed, the linearized problem of (\ref{bsd1})-(\ref{bsd3}) around $u=1$ is given by
\bgee
&&\phi_t=\Delta \phi+(p-1)\phi-r\gamma \avint \phi, \quad \mbox{in}\quad\Omega \times (0, T),\label{ln1}\\
&&\frac{\partial \phi}{\partial\nu}=0, \quad\mbox{on}\quad\partial\Om, \label{ln2}
\egee
and can be written in the form of an evolution equation in $X=L^2(\Omega)$ as
\[ \frac{d\phi}{dt}=-A\phi. \]
Here the generator $A$ is a self-adjoint operator associated with the bi-linear form (see Kato \cite{kato})
\[ a(\phi,w)=\avint\left(\nabla \phi\cdot \nabla w+(1-p)\phi w \right)\ dx+r\gamma \avint \phi\cdot\avint w, \quad \phi,w\in V=H^1(\Omega). \]
Now for $\phi=w$ we derive
\begin{eqnarray*}
a(\phi,\phi) & = & \Vert \nabla \phi\Vert_2^2+(1-p)\avint \phi^2+r\gamma \left(\avint \phi\right)^2 \\
& = & \sum_{j=1}^\infty\mu_j^2\vert (\phi,\varphi_j)\vert^2+\sum_{j=1}^\infty(1-p)\vert (\phi, \varphi_j)\vert^2+r\gamma\vert (\phi,\varphi_1)\vert^2 \\
& = & (1-p+r\gamma)\vert (\phi, \varphi_1)\vert^2+\sum_{j=2}^\infty(\mu_j^2+1-p)\vert (\phi, \varphi_j)\vert^2,
\end{eqnarray*}
where $0=\mu_1<\mu_2\leq \cdots \rightarrow\infty$ denote the eigenvalues of $-\Delta$ associated with the Neumann boundary condition, and $\varphi_j$ is the corresponding $j$-th eigenfunction normalized by $\Vert \varphi_j\Vert_2=1$.
Note that under the Turing condition (\ref{turing2}), the linearized instability of the steady-state solution $u=1$ arises if and only if $\mu^2_2<p-1.$ The latter suggests that under condition (\ref{turing2}) a Turing instability phenomenon should be anticipated, which in particular  as shown in Theorem \ref{thm:1.7},  this Turing instability is exhibited in the form of a diffusion-driven blow-up.

\section{Main Results}\label{sec:mr}
In the current section we present our main results and we prove them in the following sections.

The first observation regarding the solution of \eqref{bsd1}-\eqref{bsd3} is that it never quenches in finite time. Indeed, the following holds
\begin{proposition}\label{lbd}
Each $T>0$ admits $C_T>0$ such that
\begin{equation} \label{jg0}
u(x,t)\geq C_T\quad \mbox{in $\Om\times [0,T)$}.
\end{equation}
\end{proposition}
\begin{proof}
By maximum principle and comparison theorems we obtain that $u(x,t)>0$ and $u(x,t)\geq \tilde u(t)$, where $\tilde{u}=\tilde{u}(t)$ solves the following
\[ \frac{d\tilde{u}}{dt}=-\tilde{u}\quad \mbox{in $(0, T)$},\quad \tilde{u}(0)=\tilde{u}_0\equiv \inf_{\Om} u_0(x)>0. \]
Therefore we obtain (\ref{jg0}) with $C_T=\tilde{u}_0e^{-T}$.
\end{proof}
Due to Proposition \ref{lbd} the following alternatives are left; blow-up in finite time indicated by $T<+\infty$, blow-up in infinite time, quenching in infinite time, and global-in-time compact orbit in $C(\overline{\Omega})$. In fact, by (\ref{jg0}) and the parabolic regularity the existence time of the classical solution $u=u(\cdot, t)$ of \eqref{bsd1}-\eqref{bsd3} in $0\leq t \leq \tilde T$ is estimated below by $\tilde T$ and $\Vert u_0\Vert_\infty$ (see, e.g., \cite{r84}). Then there holds
\begin{equation}
T<+\infty \ \Rightarrow \ \lim_{t\uparrow T}\Vert u(\cdot,t)\Vert_\infty=+\infty
  \label{eqn:1.13h}
\end{equation}
as in (8.25) of \cite{ss12}.
Now finite-time blow-up actually arises under the conditions of the following proposition.
\begin{proposition}\label{ats1}
Assume that $p \geq r$ and the Turing condition \eqref{turing2} holds then
\begin{equation}
\bar{u}\geq \max\{ 1, \bar{u}_0\}.
 \label{eqn:1.13}
\end{equation}
Whereas if anti-Turing condition \eqref{aturing} holds and $\bar{u}_0>1$, then finite-time blow-up occurs, i.e. $T<+\infty.$
 \label{lem:1.2}
\end{proposition}
\begin{proof}
Since $p>1$ and $p\geq r$, there is $r\leq \mu\leq p$ satisfying $s\geq 1$. Then we obtain
\[ \avint u^p\geq \left( \avint u^{\mu}\right)^{\frac{p}{\mu}}, \quad \left( \avint u^r\right)^\gamma\leq \left( \avint u^\mu \right)^{\frac{r}{\mu}\gamma} \]
via H\"{o}lder's inequality and hence
\begin{equation}
\displaystyle\frac{d\bar{u}}{dt}= -\bar{u}+\displaystyle{\frac{\avint u^p}{\left(\avint u^r\right)^{\gamma}}}\geq -\bar{u}+\left(\avint u^\mu\right)^{\frac{p-r\ga}{\mu}}\geq -\bar{u}+\bar{u}^{p-r\gamma}.
 \label{eqn:1.14}
\end{equation}
In case $p-r\gamma<1$ then the differential inequality  (\ref{eqn:1.14}) implies \eqref{eqn:1.13}. Whilst, in the complementary case $p-r\gamma>1$ again by virtue of (\ref{eqn:1.14}) we derive that $\bar{u}$ blows up in finite time provided $\bar{u}_0>1,$ and hence $u$ does so.
\end{proof}
\begin{rem}
Proposition \ref{ats1} illustrates that in that case the qualitative behaviour of the full system \eqref{gm1}-\eqref{gm4} and those of the non-local problem \eqref{bsd1}-\eqref{bsd3} is quite different. It should be pointed out that the anti-Turing condition \ref{aturing} the full system does not exhibit any instability, whilst an instability emerges when Turing condition \eqref{turing2} is imposed.
\end{rem}
Here by quenching in infinite time we mean
\bge\label{mkl1}
T=+\infty \quad\mbox{and}\quad \liminf_{t\uparrow+\infty}\Vert u(\cdot,t)\Vert_r=0\quad\mbox{for some}\quad r>1.
\ege
We note that property \eqref{mkl1} does not arise neither for the original system (\ref{gm1})-(\ref{gm4}) nor for the shadow system (\ref{sd1})-(\ref{sd2}), as it can be concluded by the classification of the homogeneous orbits given in  \cite{nst06}. Therefore, our first main result, see Theorem \ref{eiq} below,  which concerns the exclusion of infinite time quenching for the solutions of (\ref{bsd1})-(\ref{bsd3}) is in agreement with what is observed in systems (\ref{gm1})-(\ref{gm4}) and (\ref{bsd1})-(\ref{bsd3}).

Henceforth, $C$ and $c$ denote large and small positive constants independent of $t$, respectively.
\begin{theorem}\label{eiq}
There is $\delta_0>0$ such that any $0<\delta\leq \delta_0$ admits the estimate
\begin{equation}
\avint u^{-\delta}\leq C\quad\mbox{for any}\quad t>0,
 \label{eqn:1.16h}
\end{equation}
 \label{thm:1.1}
\end{theorem}
where constant $C$ is independent of $t.$
\begin{rem}\label{ts3}
In fact, inequality (\ref{eqn:1.16h}) implies $\displaystyle{ \avint u^{\delta}\geq c=C^{-1}}$ and then
\begin{equation}
\avint u^r\geq \left(\avint u^\delta \right)^{r/\delta}\geq c^{r/\delta}>0\quad\mbox{for any}\quad t>0,
 \label{eqn:1.17}
\end{equation}
follows by Jensen's inequality taking $\delta\leq r,$ where again $c$ is independent of time $t.$ Consequently, relation \eqref{eqn:1.17} guarantees that the nonlinear term of non-local equation \eqref{bsd1} stays away from zero and therefore the solution $u$ is bounded away from zero as well.
\end{rem}
\begin{rem}
Remark \ref{ts3} is interpreted in biological context as follows: the activator can never be diminished.
\end{rem}
Next we focus on the special case $r=p+1$. In this case problem (\ref{bsd1})-(\ref{bsd3}) admits a variational structure which is not the case for the original system (\ref{gm1})-(\ref{gm4}). In particular for $r=p+1$ problem (\ref{bsd1})-(\ref{bsd3}) has a Lyapunov functional of the form
\[ J(u)= \displaystyle\frac{1}{2}\left( \Vert \nabla u\Vert_2^2+\Vert u\Vert_2^2\right)-\displaystyle\frac{1}{(p+1)(1-\gamma)}\left(\displaystyle\avint u^{p+1} \, dx\right)^{1-\gamma}, \]
since along any solution trajectory there holds
\begin{equation}\label{27}
\displaystyle\frac{d}{dt} J(u(t))=-\Vert u_t\Vert_2^2 \leq  0.
\end{equation}
Note also that in that case Turing condition (\ref{turing2}) is reduced to
\be\label{mk4}
 \gamma>\frac{p-1}{p+1}.
\ee
Next under anti-Turing condition and via the  double well potential method (see \cite{sattinger68, Levine73}, we obtain the following:
\begin{theorem}
Let $r=p+1$ and  $\gamma<\min\{ 1, \frac{p-1}{p+1}\}.$ If $J(u_0)\leq 0$ then finite-time blow-up occurs, i.e.  $T<+\infty$.
 \label{thm:1.2}
\end{theorem}
\begin{rem}
Theorem \ref{thm:1.2} can be interpreted in the biological context as follows: if the activator's initial concentration is large and  its suppression by the inhibitor is rather small (since $0<\ga<1$) then naturally the activator's growth becomes unlimited.
\end{rem}
On the other hand, under the Turing condition \eqref{mk4} we derive the following:
\begin{theorem}\label{thm:1.2a}
Let $N\geq 3$ and $r=p+1$. If $\frac{p-1}{p+1}<\ga<1$ and $1<p<\frac{N+2}{N-2}$ then a global-in-time solution exists, i.e.
\begin{equation}
T=+\infty\quad\mbox{and}\quad \sup_{(0,T)}\Vert u(\cdot,t)\Vert_\infty\leq C.
 \label{eqn:1.16}
\end{equation}
 \label{prop1}
\end{theorem}
\begin{rem}
Theorems \ref{thm:1.2} and \ref{thm:1.2a} indicate that in case where $r=p+1$ then there is a discrepancy  between the behaviour of the full system \eqref{gm1}-\eqref{gm4} and those of the non-local problem \eqref{bsd1}-\eqref{bsd3}. Indeed, under the anti-Turing condition the full system does not exhibit any instability, whilst an instability occurs when Turing condition \eqref{mk4} holds.
\end{rem}
Turing condition \eqref{mk4} further implies that the solution orbit for problem (\ref{bsd1})-(\ref{bsd3}) is compact in $C(\overline{\Omega})$ and the $\omega$-limit set
\[ \omega(u_0)=\{u_\ast\in C(\overline{\Omega}) \mid \mbox{$\exists t_k\uparrow+\infty$ s.t. }\lim_{k\rightarrow\infty}\Vert u(\cdot,t_k)-u_\ast\Vert_\infty=0\} \]
of this orbit is nonempty, connected, compact, and lies in the set of stationary solutions, which are defined as the solutions of the following problem
\begin{equation}
-\Delta u_\ast +u_\ast =\displaystyle\frac{u_\ast^p}{\Big(\displaystyle\avint u_\ast^{p+1} \Big)^{\gamma}}, \ u_\ast>0 \quad\mbox{in}\quad\Omega, \quad \displaystyle\frac{\partial u_\ast}{\partial \nu}=0\quad\mbox{on}\quad \partial\Omega.
 \label{ssd2}
\end{equation}
Concerning (\ref{ssd2}), existence of stable spiky stationary solutions is known (see the survey paper by Wei \cite{wei08}) and thus formation of Turing patterns converging to these spiky solutions is guaranteed as long as \eqref{mk4} holds.


In the following, global-in-time existence of the solution is obtained via a priori estimates of some $L^{\ell}$-norms of solution $u(x,t).$ These a priori estimates hold in a parameter range which implies condition $\frac{p-1}{r}<1$ which, as mentioned earlier, guarantees the global-in-time existence of the solution to the original model (\ref{gm1})-(\ref{gm4}).
\begin{theorem}
If $\frac{p-1}{r}<\min\{1, \frac{2}{N}, \frac{1}{2}(1-\frac{1}{r})\}$ and $0<\gamma<1$, then problem (\ref{bsd1})-(\ref{bsd3}) has a global-in-time solution, i.e. (\ref{eqn:1.16}) holds.
 \label{thm:1.3}
\end{theorem}
\begin{rem}
The result of Theorem \ref{thm:1.3} is in agreement with the global-in-time existence result obtained in \cite{j06} for the full system, and so in that case \eqref{gm1}-\eqref{gm4} and \eqref{bsd1}-\eqref{bsd3} share the same dynamics.
\end{rem}
Now we consider the following $L^{\ell}-$norms , $\ell>0,$ of solution $u(x,t)$
 \begin{equation}
\zeta(t)=\avint u^r\,dx, \quad z(t)=\avint u^{-p+1+r}\,dx, \quad w(t)=\avint u^{p-1+r}\,dx.
 \label{eqn:1.21}
\end{equation}
By choosing proper initial data and using  phase plane analysis, we can actually derive estimates of $\zeta(t), z(t)$ and $w(t),$  see section \ref{sec:4}, identifying also some invariant regions in the plane. In particular, our results can be expressed as follows:
\begin{theorem}\label{obu1}
Let $0<\gamma<1$, $r\leq 1$, and $\frac{p-1}{r}>1$. Assume, furthermore that either (1) $w(0)<\zeta(0)^{1-\gamma}$, or (2) $\frac{p-1}{r}\geq 2$ and $w(0)<1.$ Then finite-time blow-up occurs for problem (\ref{bsd1})-(\ref{bsd3}), i.e. $T<+\infty.$
\end{theorem}
\begin{rem}
Since $\frac{p-1}{r}>1> \gamma$ is assumed, Theorem \ref{obu1} is associated with the finite-time blow-up under anti-Turing condition and is in agreement with the blow-up result  \cite[Theorem 2]{msy95}. This is actually an indication, under condition $\frac{p-1}{r}>1> \gamma$,the qualitative behaviour of the full system \eqref{gm1}-\eqref{gm4} and those of the non-local problem (\ref{bsd1})-(\ref{bsd3}) is quite similar.
\end{rem}
\begin{rem}
The biological interpretation of Theorem \ref{obu1} is as follows: a large initial concentration for the activator combined with small {\it net cross-inhibition index} can lead to its unlimited growth.
\end{rem}
\begin{theorem}
 Let $\gamma>1$, $r\geq 1$ and $\frac{p-1}{r}<1$. Assume, further that $w(0)<\zeta(0)^{1-\gamma}$ and $\zeta(0)^{1+\gamma}>z(0)$.  Then problem (\ref{bsd1})-(\ref{bsd3}) has a global-in-time solution, i.e. (\ref{eqn:1.16}) holds.
 \label{thm:1.6}
\end{theorem}
\begin{rem}
Theorem \ref{thm:1.6}, on the contrary, deals with the case of global-in-time existence under Turing condition and it is also in agreement with Jiang's result in \cite{j06}. Consequently under assumptions of Theorem \ref{thm:1.6} both the full system \eqref{gm1}-\eqref{gm4} as well as the  non-local problem (\ref{bsd1})-(\ref{bsd3}) ensemble the same long-time behaviour.
\end{rem}
In section \ref{prl} it has been already specified through linear stability analysis that under the Turing condition \eqref{turing2} the stable solution $\bar{u}=1$ of \eqref{ts1} destabilises as a solution of (\ref{bsd1})-(\ref{bsd3}). The next result shows that the preceding {\it DDI} phenomenon is realised in the form of diffusion-induced blow-up.
\begin{theorem}
Let $N\geq 3,\; \max\{1,\frac{N}{N-2}\}\leq r\leq p$ and $\frac{2}{N}<\frac{p-1}{r}<\gamma$, then there is a family of radially symmetric blowing up solutions corresponding to a family of spiky initial data.
 \label{thm:1.7}
\end{theorem}
\begin{rem}
Usually {\it DDI} phenomena are connected with pattern formation. The same happens in the case of the diffusion-induced blow-up provided by Theorem \ref{thm:1.7}. The form of the destabilising patterns are determined in section \ref{sec:7}, see particularly Remark \ref{nkl7}.
\end{rem}
\section{Proof of Theorem \ref{thm:1.1}}\label{sec:qi}

Set $\chi=u^{\frac{1}{\al}}$, $\al\neq 0$, then $\chi$ satisfies
\bge
&&\al \chi_t=\alpha\left(\Delta \chi+4 (\al-1) \vert\nabla \chi^{\frac{1}{2}}\vert^2\right)-\chi+f \quad \mbox{in}\quad \Om\times (0,T),\label{ob1}\\
&& \frac{\pl \chi}{\pl \nu}=0, \quad \mbox{on}\quad \partial\Om\times (0,T),\label{ob2a}\\
&&\chi(x,0)=u_0^{\frac{1}{\al}}(x),\quad\mbox{in}\quad \Omega,\label{ob2}
\ege
with
\begin{equation}
f=\frac{u^{p-1+\frac{1}{\al}}}{\left(\avint u^r\right)^{\ga}}
 \label{ob3}.
\end{equation}
Averaging (\ref{ob1}) over $\Om$, we obtain
\begin{equation}\label{ob4}
\alpha\frac{d}{dt}\avint \chi+4\alpha(1-\alpha)\avint\vert \nabla \chi^{\frac{1}{2}}\vert^2+\avint \chi=\avint f
\end{equation}
and hence
\bge\label{ts2}
\frac{d}{dt}\avint \chi+4(1-\alpha)\avint\vert\nabla \chi^{\frac{1}{2}}\vert^2+\frac{1}{\alpha}\avint \chi\leq 0
\ege
for $\alpha<0$ since also $f>0.$ Letting now $\delta=-\frac{1}{\alpha}$ we have
\[ \frac{d}{dt}\avint \chi+4(1+\delta^{-1})\avint\vert \nabla \chi^{\frac{1}{2}}\vert^2 \leq \delta \avint \chi. \]
Since Poincar\'e-Wirtinger's inequality reads
\bgee
\Vert \nabla w\Vert_2^2\geq \mu_2\Vert w\Vert_2^2,\quad\mbox{for any}\quad w\in H^1(\Om),
\egee
where $\mu_2$ is the second eigenvalue of the Laplace operator associated with Neumann boundary conditions, then applied for $w=\chi^{\frac{1}{2}}$ to \eqref{ts2} entails
\be\label{mk5}
\frac{d}{dt}\avint \chi+c\avint \chi\leq 0
\ee
for $0<\delta\ll 1$. Differential inequality  \eqref{mk5} implies that $\chi(t)\leq C<\infty$ for any $t>0$ and thus (\ref{eqn:1.16h}) follows by the fact that $\chi=u^{-\delta}.$

\section{Proof of Theorems \ref{thm:1.2} and \ref{thm:1.2a}}\label{sec:2}

Throughout the current section, we consider $r=p+1$.

\begin{proof}[Proof of Theorem \ref{thm:1.2}]
Since $J(u_0)\leq 0$, then via the dissipation relation \eqref{27} we derive  $J(u(t))\leq 0$ for any $0<t<T.$

We also have
\[ \frac{d}{dt}\Vert u\Vert_2^2=-2I(u) \]
where
\begin{eqnarray*}
I(u) & = & \Vert \nabla u\Vert_2^2+\Vert u\Vert_2^2-\frac{\avint u^{p+1}}{\left( \avint u^r\right)^\gamma} = \Vert \nabla u\Vert_2^2+\Vert u\Vert_2^2-\left( \avint u^{p+1}\right)^{1-\gamma} \\
& = & 2J(u)+\left( \frac{2}{(p+1)(1-\gamma)}-1\right)\left( \avint u^{p+1}\right)^{1-\gamma} \\
& \leq & -\left( 1-\frac{2}{(p+1)(1-\gamma)}\right)\Vert u\Vert_{p+1}^{(p+1)(1-\gamma)}.
\end{eqnarray*}
Since $0<\gamma<\min\{ 1, \frac{p-1}{p+1}\}$, there holds that $(p+1)(1-\gamma)>2,$ and thus by virtue of H\"{o}lder's inequality we can find $\alpha>0$ such that
\bge\label{ts4}
\frac{d}{dt}\Vert u\Vert_2^2\geq c\Vert u\Vert_2^{2+\alpha},
\ege
since also $p>1.$ Now \eqref{ts4} entails that $\Vert u\Vert_2^2$ blows up in finite time since $u_0(x)>0$ and thus $u$ exhibits a finite-time blow-up as well.
\end{proof}

\begin{proof}[Proof of Theorem \ref{prop1}]
In this case we have $0<\gamma<1$ and $(p+1)(1-\gamma)<2$. Dissipation relation  \eqref{27} suggests
\bge\label{ts5}
\frac{1}{2}(\Vert \nabla u\Vert_2^2+\Vert u\Vert_2^2)\leq J(u_0)+\frac{1}{(p+1)(1-\gamma)}\left( \avint u^{p+1}\right)^{1-\gamma}.
\ege
Furthermore, Sobolev's and Young's inequalities entail
\bge\label{ts6}
\left( \avint u^{p+1}\right)^{1-\gamma}=\Vert u\Vert_{p+1}^{(p+1)(1-\gamma)}\leq \frac{1}{4}\Vert u\Vert_{H^1}^2+C,
\ege
recalling $(p+1)(1-\gamma)<2.$ Combining now \eqref{ts5} with \eqref{ts6} we derive the estimate
\begin{equation}
\Vert u(\cdot,t)\Vert_{H^1}\leq C\quad\mbox{for}\quad 0<t<T.
 \label{eqn:1.26}
\end{equation}
Now $u$ satisfies
\bgee
&& u_t=\Delta u-u+a(t)u^p, \quad\mbox{in}\quad \Om\times(0,T),\\
&& \frac{\partial u}{\partial\nu}=0,\quad\mbox{on}\quad \partial\Om\times(0,T),\\
&& u(x,0)=u_0(x)>0,\quad\mbox{in}\quad\Om,
\egee
with $$0\leq a(t)=\frac{1}{\left(\avint u^{p+1}\right)^{\gamma}}\leq C<+\infty,$$ due to \eqref{eqn:1.17}. Then, letting $A$ to be $-\Delta+1$ with the homogeneous Neumann boundary condition, we use
\[ u(\cdot,t)=e^{-tA}u_0+\int_0^te^{-(t-s)A}a(s)u(\cdot,s)^p \ ds  \]
to apply a standard bootstrap argument. In fact, it follows that (\ref{eqn:1.16}) from (\ref{eqn:1.26}), $1<p<\frac{N+2}{N-2},$ and the proof of \cite[Lemma 8.1]{ss12}.
\end{proof}

\begin{rem}
In the case where $\gamma=\frac{p-1}{p+1}$ and $1<p<\frac{N+2}{N-2}$, we have always $T=+\infty$, whilst infinite-time blow-up, i.e. $\lim_{t\uparrow +\infty}\Vert u(\cdot,t)\Vert_\infty=+\infty,$ may occur. In fact, by the proof of Theorem \ref{prop1} we have
\[ \Vert u(\cdot,t)\Vert_{H^1}\leq C(1+t) \quad \mbox{for $0<t<T$}, \]
and then by virtue of Sobolev's imbedding we obtain $\Vert u(\cdot,t)\Vert_\infty\leq C_T,$  which entails $T=+\infty$ by the parabolic regularity. Furthermore, in case where $J(u_0)<0$  we derive
\[ \frac{d}{dt}\Vert u\Vert_2^2\geq -2J(u_0)>0, \]
by the proof of Theorem \ref{thm:1.2}, and then it follows that $\lim_{t\uparrow\infty}\Vert u(\cdot,t)\Vert_2=+\infty.$ The latter implies that $\lim_{t\uparrow+\infty}\Vert u(\cdot, t)\Vert_\infty=+\infty$ and thus infinite-time blow-up occurs in that case.
\end{rem}

\section{Proof of Theorem \ref{thm:1.3}}\label{sec:3}
We assume $\frac{p-1}{r}<\min\{ 1, \frac{2}{N}, \frac{1}{2}(1-\frac{1}{r})\}$ and $0<\gamma<1$. We also consider $N\geq 2$ since the complementary case $N=1$ is simpler.

Since $p>1$, the above assumption implies $\frac{p-1}{r}<\frac{2}{N}$ and $r>p$. Then there holds that
\[ 0<\frac{1}{r-p+1}<\min\left\{ 1, \frac{1}{p-1}\cdot\frac{2}{N-2}, \frac{1}{1-p+r\gamma}\right\}, \]
since also $0<\gamma<1$.

Choosing  $\frac{1}{r-p+1}<\alpha<\min\{ 1, \frac{1}{p-1}\cdot\frac{2}{N-2}, \frac{1}{1-p+r\gamma} \}$, we have
\[ \max\left\{ \frac{N-2}{N}, \frac{1}{\alpha r}\right\}<\frac{1}{-\alpha+1+\alpha p}, \]
and hence there is $\beta>0$ such that
\begin{equation}
\max\left\{ \frac{N-2}{N}, \frac{1}{\alpha r}\right\} <\frac{1}{\beta}<\frac{1}{-\alpha+1+\alpha p}<2,
 \label{eqn:15}
\end{equation}
which also satisfies
\begin{equation}
\frac{\beta}{\alpha r}<1<\frac{\beta}{-\alpha+1+\alpha p}.
 \label{eqn:16}
\end{equation}
Note that for $f$ defined by \eqref{ob3} holds
\[ \avint f=\frac{\avint u^{p-1+\frac{1}{\alpha}}}{\left( \avint u^r\right)^\gamma}= \frac{\avint \chi^{-\alpha+1+\alpha p}}{\left( \avint\chi^{\alpha r}\right)^\gamma}. \]
By virtue of (\ref{eqn:16})
\[ \avint \chi^{-\alpha+1+\alpha p}\leq \left( \avint \chi^\beta\right)^{\frac{-\alpha+1+\alpha p}{\beta}}\quad\mbox{and}\quad \left( \avint \chi^{\alpha r}\right)^\gamma \geq \left( \avint \chi^\beta\right)^{\frac{\alpha r}{\beta}\cdot\gamma} \]
and thus
\begin{equation}
\avint f\leq \left( \avint \chi^\beta\right)^{\frac{-\alpha+1+\alpha p-\alpha r\gamma}{\beta}}=\Vert \chi^{\frac{1}{2}}\Vert_{2\beta}^{2(1-\sigma)}
 \label{eqn:3.3}
\end{equation}
with $0<\sigma=\alpha\{1-p+r\gamma\}<1$, recalling $\frac{p-1}{r}<\gamma$ and $\alpha<\frac{1}{1-p+r\gamma}$.

Now since $1<2\beta<\frac{2N}{N-2}$ holds due to (\ref{eqn:15}), then Sobolev's and Young's inequalities entail
\[ \frac{d}{dt}\avint \chi+c\Vert \chi^{\frac{1}{2}}\Vert_{H^1}^2\leq C,\quad 0<t<T, \]
using also $0<\alpha<1$, and in particular,
\[ \avint \chi\leq C, \quad\mbox{for any}\quad 0<t<T. \]
Since $\frac{1}{\alpha}$ can be chosen to be close to $r-p+1$, we have
\begin{equation}
\Vert u(\cdot,t)\Vert_q\leq C_q, \quad 0<t<T,\quad\mbox{for any}\quad 1\leq q<r-p+1,
 \label{eqn:15h}
\end{equation}
taking into account that $\chi=u^{\frac{1}{\alpha}}.$

Since $\frac{p-1}{r}<\frac{1}{2}(1-\frac{1}{r})$ implies $\frac{r-p+1}{p}>1,$ then if there is $a>1$ such that
\bge\label{ikl1}
\Vert u(\cdot,t)\Vert_q\leq C_q, \quad 0<t<T,\quad\mbox{for any}\quad 1\leq q<a(r-p+1),
\ege
by virtue of the semigroup estimate \eqref{eqn:1.2} inequality \eqref{ikl1} can be extended for any $q\geq 1$ as long as $\frac{N}{2}(\frac{1}{\ell}-\frac{1}{q})<1$. Therefore, we obtain
\bge\label{psl2}
\Vert u(\cdot,t)\Vert_q\leq C_q, \quad 0<t<T,\quad\mbox{for any}\quad 1\leq q<a_1(r-p+1),
\ege
and for $a_1>0$ defined by
\bge\label{psl1}
 \frac{1}{a_1}=\frac{1}{a}-\frac{2}{N}\cdot\frac{r-p+1}{p},
\ege
as long as the right-hand side of \eqref{psl1} is positive, otherwise  $q=\infty$ into relation \eqref{psl2}. We eventually obtain (\ref{eqn:1.16}), and the proof is complete.

\section{Proof of Theorems \ref{obu1} and \ref{thm:1.6}}\label{sec:4}

Let $(\zeta, z,w)=(\zeta(t), z(t), w(t))$ defined by (\ref{eqn:1.21}) then, by virtue of H\"{o}lder's inequality we have
\begin{equation} \label{ob5}
wz\geq \zeta^2.
\end{equation}
\begin{proof}[Proof of Theorem \ref{obu1}]
We first consider $r\leq 1<\frac{p-1}{r}$ and $0<\gamma<1$. Since $r\leq 1$ then \eqref{ob4} for $\alpha=\frac{1}{r}$ yields
\begin{eqnarray}
\frac{1}{r} \frac{d \zeta}{dt}& = & \frac{4}{r}\left(\frac{1}{r}-1\right)\avint \vert\nabla u^{\frac{r}{2}}\vert^2-\zeta+\frac{z}{\zeta^{\gamma}}\quad\mbox{for}\quad 0<t<T,
 \label{ob6}
\end{eqnarray}
and taking \eqref{ob5} into account we derive
\begin{eqnarray}
\frac{1}{r} \frac{d \zeta}{dt}\geq  -\zeta+\frac{\zeta^{2-\gamma}}{w} =\frac{\zeta}{w}\left(-w+\zeta^{1-\gamma}\right)\quad\mbox{for}\quad 0<t<T.
 \label{ob6a}
\end{eqnarray}
Furthermore, since $\frac{p-1}{r}>1$ then \eqref{ob4} for $\alpha=\frac{1}{-p+1+r}$ reads
\be\label{mk2}
\alpha\frac{dw}{dt}=4\alpha(\alpha-1)\avint \vert \nabla u^{\frac{1}{2\alpha}}\vert^2-w+\zeta^{1-\gamma},\quad\mbox{for}\quad 0<t<T,
\ee
which, since $\alpha=\frac{1}{-p+1+r}<0,$  implies
\bgee
\alpha\frac{dw}{dt}\geq -w+\zeta^{1-\gamma},\quad\mbox{for}\quad 0<t<T,
\egee
or equivalently
\begin{equation}
\frac{1}{p-1-r} \frac{dw}{dt}\leq w-\zeta^{1-\gamma},\quad\mbox{for}\quad 0<t<T.
 \label{ob7}
\end{equation}
The condition $0<\gamma<1$, entails that the curve
\begin{equation}
\Gamma_1: w=\zeta^{1-\gamma}, \ \zeta>0,
 \label{eqn:5.4}
\end{equation}
is concave in the $w\zeta-$plane, with its endpoint at the origin $(0,0)$. Relations  (\ref{ob6a}) and (\ref{ob7}) imply that the region $\mathcal{R}=\{ (\zeta,w) \mid w<\zeta^{1-\gamma}\}$ is invariant for the system \eqref{ob6}, \eqref{mk2}, i.e. if $(\zeta(0),w(0))\in \mathcal{R}$ then $(\zeta(t),w(t))\in \mathcal{R}$ for any $t>0.$ Furthermore,  $\zeta=\zeta(t)$ and $w=w(t)$ are increasing and decreasing on $\mathcal{R},$ respectively.

In case  $w(0)<\zeta(0)^{1-\gamma}$, then
\[ \frac{dw}{dt}<0, \ \frac{d\zeta}{dt}>0, \quad\mbox{for}\quad 0\leq t<T, \]
and thus,
\[ \frac{1}{w}-\frac{1}{\zeta^{1-\gamma}}\geq \frac{1}{w(0)}-\frac{1}{\zeta(0)^{1-\gamma}}\equiv c_0>0, \quad\mbox{for}\quad 0\leq t<T.\]
Therefore by virtue of \eqref{ob6a}
\begin{equation}
\frac{1}{r}\frac{d\zeta}{dt}\geq -\zeta+\frac{\zeta^{2-\gamma}}{w}=\zeta^{2-\gamma}\left(\frac{1}{w}-\frac{1}{\zeta^{1-\gamma}}\right) \geq c_0\zeta^{2-\gamma}, \quad 0\leq t<T.
 \label{eqn:5.3}
\end{equation}
Since $2-\gamma>1$ then (\ref{eqn:5.3}) implies that $\zeta(t)$ blows up in finite time
\bgee
t_1\leq \hat{t}_1\equiv\frac{(\zeta(0))^{\ga-1}}{(1-\ga)c_0 r},
\egee
and using the inequality
\bgee
\zeta(t)=\avint u^r\,dx\leq \|u(\cdot,t)\|^r_{\infty}
\egee
we conclude that $u(x,t)$ blows up in finite time $T\leq t_1$ as well.

We consider now the second case when  $\frac{p-1}{r}\geq 2$ and thus $q=\frac{p-1-r}{r}\geq 1.$ Then by virtue of Jensen's inequality
\bgee
\avint u^r\cdot\left( \avint (u^{-r})^q\right)^{\frac{1}{q}}\geq \avint u^{r}\cdot\avint u^{-r}\geq 1,
 \label{eqn:5.5}
\egee
and thus $\zeta^{\frac{1}{r}}\geq w^{-\frac{1}{p-1-r}}$ which  entails
\be\label{mk1}
w\geq \zeta^{-\frac{p-1-r}{r}}=\zeta^{1-\frac{p-1}{r}}.
\ee
In addition, inequality $\frac{p-1}{r}\geq 2$ implies that the curve
\[ \Gamma_2: w=\zeta^{1-\frac{p-1}{r}}, \ \zeta>0, \]
is convex and approaches $+\infty$ and $0$ as $\zeta\downarrow 0$ and $\zeta\uparrow+\infty$, respectively. The crossing of $\Gamma_1$ and $\Gamma_2$ is the point $(\zeta,w)=(1,1)$, and therefore $w(0)<1$ combined with \eqref{mk1} imply $w(0)<\zeta(0)^{1-\gamma}$. Consequently, the second case is reduced to the first one and again the occurrence of finite-time blow-up is established.
\end{proof}
\begin{rem}
The existence of the invariant region $\mathcal{R}=\{ (\zeta,w) \mid w<\zeta^{1-\gamma}\}$ for the system \eqref{ob6}, \eqref{mk2} entails that if the consumption of the activator cannot be suppressed initially then this can lead to its unlimited growth.
\end{rem}
\begin{proof}[Proof of Theorem \ref{thm:1.6}]
We first note that under the assumption $r\geq 1$ relation \eqref{ob6} entails
\begin{equation}
\frac{1}{r}\frac{d\zeta}{dt}\leq -\zeta+\frac{z}{\zeta^\gamma},\quad\mbox{for}\quad 0\leq t<T.
 \label{eqn:5.6}
\end{equation}
Furthermore, since $\alpha=\frac{1}{p-1-r}<0$ results from $\frac{p-1}{r}<1$, then \eqref{mk2} implies
\bgee
\alpha\frac{dw}{dt}\geq -w+\zeta^{1-\gamma},\quad\mbox{for}\quad 0\leq t<T,
 \egee
or equivalently
\begin{equation}
\frac{1}{-p+1+r}\frac{dw}{dt}\leq w-\zeta^{1-\gamma},\quad\mbox{for}\quad 0\leq t<T.
 \label{eqn:5.8}
\end{equation}

We claim that the assumption $\zeta(0)^{1+\gamma}>z(0)$ yields that  $\zeta(t)^{1+\gamma}>z(t)$ for any $0\leq t<T.$ Indeed, let us assume there exists $t_0>0$ such that
\[ \zeta(t)^{1+\gamma}>z(t), \ 0\leq t<t_0, \quad\mbox{and}\quad \zeta(t_0)=z(t_0). \]
Then we obtain
\bge
\frac{d\zeta}{dt}<0, \quad\mbox{for}\quad 0\leq t<t_0\quad\mbox{and}\quad w(t_0)\geq \zeta(t_0)^{1-\gamma},
 \label{eqn:5.9}
\ege
by virtue of \eqref{eqn:5.6} and (\ref{ob5}).

On the other hand,  $w(0)<\zeta(0)^{1-\gamma},$  due to (\ref{eqn:5.8}), entails
\bgee
\frac{d w}{dt}<0, \quad\mbox{for}\quad 0\leq t<t_0.
 \label{eqn:5.9a}
\egee
Consequently, since also $\gamma>1$ then the curve $(w(t), \zeta(t))$ for $0\leq t\leq t_0$, remains in the region $w<\zeta^{1-\gamma}$ and hence $w(t_0)<\zeta(t_0)^{1-\gamma}$, which contradicts the second inequality of (\ref{eqn:5.9}).

Thus it follows that
\[ \frac{d\zeta}{dt}<0, \ \frac{dw}{dt}<0, \quad\mbox{for}\quad 0\leq t<T, \]
and in particular, we have
\[ \Vert u(\cdot,t)\Vert_{p-1+r}\leq C,\quad\mbox{for}\quad 0\leq t<T. \]
Since $r\geq 1$ implies $\frac{p-1+r}{p}\geq 1$, we obtain (\ref{eqn:1.16}) by the same  bootstrap argument used at the end of the previous section.
\end{proof}

\section{Proof of Theorem \ref{thm:1.7}}\label{sec:5}

In the current  section we restrict ourselves to the radial case $\Om=B(0,1)$ and we also consider $N\geq 3.$ Then the solution  of \eqref{bsd1}-\eqref{bsd3} is radial symmetric, that is $u(x,t)=u(\rho,t)$ for $0\leq \rho=\vert x\vert<1.$

We regard, as in \cite{hy95}, spiky initial data of the form
\begin{equation}
u_0(\rho)=\lambda\varphi_\delta(\rho),
 \label{eqn:6.1}
\end{equation}
with $0<\lambda\ll 1$ and
\begin{equation}\label{id}
\varphi_\delta(\rho)=\begin{cases}
      \rho^{-a},&  \de\leq \rho\leq 1 \\
      \de^{-a}\left(1+\frac{a}{2}\right)-\frac{a}{2}\de^{-(a+2)}\rho^2, & 0\leq \rho<\de,
      \end{cases}
\end{equation}
for $a=\frac{2}{p-1}$ and $0<\de<1.$

It can be easily checked that $u_0(\rho)$ is decreasing, i.e.  $u'_{0}(\rho)<0,$ and thus  $\max_{\rho\in[0,1]}u_0(\rho)=u_0(0).$ Furthermore, due to the maximum principle we have that $u(\rho,t)$ is radial decreasing too, i.e. $u_{\rho}(\rho,t)<0.$

Now having specified the form of the considered initial data, Theorem \ref{thm:1.7} can be rewritten as follows:
\begin{theorem}
Let $N\geq 3,\;1\leq r\leq p$, $p>\frac{N}{N-2}$ and $\frac{2}{N}<\frac{p-1}{r}<\gamma.$ Then there is $\lambda_0>0$ with the following property: any $0<\lambda\leq \lambda_0$ admits $0<\delta_0=\delta_0(\lambda)<1$ such that any solution of problem (\ref{bsd1})-(\ref{bsd3}) with initial data of the form (\ref{eqn:6.1}) and $0<\delta\leq \delta_0$ blows up in finite time, i.e. $T<+\infty.$
 \label{thm:6.1}
\end{theorem}
We perceive that  Theorem \ref{thm:6.1} for $r=1$ is nothing but Proposition 3.3 in \cite{ln09}, which was proven using  a series of auxiliary results and inspired by an approach introduced in \cite{fmc85, hy95}. Therefore, in order to prove Theorem \ref{thm:6.1}  we are following in short the arguments presented in \cite{ln09}, and we provide any modifications where are necessary.

The next lemma is elementary and so its proof is omitted.
\begin{lemma}\label{kkl1}
The function $\phi_{\delta}$ defined by \eqref{id} satisfies the following:
\begin{enumerate}
\item[(i)]There holds that
\begin{equation}
\displaystyle{ \Delta \varphi_\delta\geq -Na\varphi_\delta^p}
 \label{eqn:6.3}
\end{equation}
in the weak sense for any $0<\delta<1$.
\item[(ii)] If $m>0$ and $N>ma$, we have
\begin{equation}\label{ine}
\avint \varphi_\delta^m=\frac{N}{N-ma}+O\left(\de^{N-ma}\right), \quad \delta\downarrow 0.
\end{equation}
\end{enumerate}
\end{lemma}
Lemma \ref{kkl1} can used to obtain some further useful estimates. Indeed, if we consider
\begin{equation}
\mu>1+r\gamma
 \label{eqn:6.15}
\end{equation}
and set
\begin{equation}
\alpha_1=\sup_{0<\delta<1}\frac{1}{\bar{\varphi}_\delta^\mu}\avint\varphi_\delta^p, \quad\mbox{and}\quad \alpha_2=\inf_{0<\delta<1}\frac{1}{\bar{\varphi}_\delta^\mu}\avint\varphi_\delta^p,
 \label{eqn:6.6}
\end{equation}
then since $p>\frac{N}{N-2}$, relation(\ref{ine}) is applicable for $m=p$ and $m=1$, and thus due to \eqref{eqn:6.15} we obtain
\bge\label{at1} 0<\alpha_1, \alpha_2<\infty. \ege
Furthermore, there holds that
\begin{equation}
d\equiv\inf_{0<\delta<1}\left( \frac{1}{2\alpha_1}\right)^{\frac{r\gamma}{p}}\left( \frac{1}{2\bar{\varphi}_\delta}\right)^{\frac{r\gamma}{p}\mu}>0.
 \label{eqn:6.7}
\end{equation}
The following auxiliary result provides a key inequality satisfied by the initial data $u_0=u_0(\vert x\vert)$ defined by \eqref{eqn:6.1}. Indeed, we have
\begin{lemma}
If $p>\frac{N}{N-2}$ and $\frac{p-1}{r}<\gamma$, there exists $\lambda_0=\lambda_0(d)>0$ such that for any $0<\lambda\leq \lambda_0$ there holds
\begin{equation}
\Delta u_0+d\lambda^{-r\gamma}u_0^p\geq 2u_0^p.
 \label{eqn:6.4}
\end{equation}
 \label{lem:6.2}
\end{lemma}

\begin{proof}
Note that inequality (\ref{eqn:6.4}) is equivalent to
\[ \Delta \varphi_\delta+d\lambda^{-r\gamma+p-1}\varphi_\delta^p\geq 2\lambda^{p-1}\varphi_\delta^p \]
which is reduced to
\[ d\lambda^{-r\gamma+p-1}\geq Na+2\lambda^{p-1} \]
due to (\ref{eqn:6.3}). Then the result follows since $\frac{p-1}{r}<\gamma$.
\end{proof}
Henceforth we fix $0<\lambda\leq \lambda_0=\lambda_0(d)$ so that (\ref{eqn:6.4}) is satisfied.  Given $0<\delta<1$, let $T_\delta>0$ be the maximal existence time of the solution to (\ref{bsd1})-(\ref{bsd3}) with initial data of the form (\ref{eqn:6.1}).

In order, to get rid off the linear dissipative term $-u$ we introduce $z=e^tu$, which then satisfies
\bge
&&z_t= \Delta z +K(t)z^p, \quad\mbox{in}\quad Q\equiv \Omega\times (0,T_\delta),\label{tbsd3}\\
&&\displaystyle\frac{\partial z}{\partial \nu}=0, \quad\mbox{on}\quad \partial \Omega\times (0,T_\delta),\label{tbsd3a}\\
&& z(x,0)=u_0(\vert x\vert),\quad\mbox{in}\quad \Omega,\quad\label{tbsd3b}
 \ege
where
\begin{equation} \label{nnt}
K(t)=\frac{e^{(1+r\gamma-p)t}}{\Big(\displaystyle\avint z^r \Big)^{\gamma}}.
\end{equation}
It is clear that $u$ blows up in finite time if and only if $z$ does so.

Due to (\ref{eqn:1.17}) we have
\begin{equation}
0< K(t)=\frac{e^{(1-p)t}}{\Big(\displaystyle\avint u^r \Big)^{\gamma}}\leq C<\infty,
 \label{eqn:6.8}
\end{equation}
thus (\ref{tbsd3}) entails
\begin{equation}
\frac{d\bar{z}}{dt}=K(t)\avint z^p
 \label{eqn:6.8h}
\end{equation}
and we finally derive the following estimate
\begin{equation}
\bar{z}(t)\geq \bar{z}(0)=\avint u_0.
 \label{eqn:6.9}
\end{equation}
Another helpful estimate of $z$ is given by the following lemma
\begin{lemma}\label{lem3}
There holds that
\begin{equation}
\rho^Nz(\rho,t)\leq \bar{z}(t) \quad\mbox{in}\quad(0,1)\times (0,T_{\de}),
 \label{eqn:6.10}
\end{equation}
and
\begin{equation}
z_\rho\left(\frac{3}{4},t\right)\leq -c, \ \ 0\leq t<T_\delta,
 \label{eqn:6.11}
\end{equation}
for any $0<\delta<1$.
\end{lemma}
\begin{proof}
Set $w=\rho^{N-1}z_{\rho}$, then $w$ satisfies
\bgee
&& \mathcal{H}[w]=0, \quad\mbox{in}\quad (0,1)\times (0,T_{\de}),\label{eqn:6.12}\\
&& w(0,t)=w(1,t)=0,\quad\mbox{for}\quad t\in(0,T_{\de}), \label{eqn:6.12a}\\
&& w(\rho,0)<0,\quad\mbox{for}\quad 0<\rho<1, \label{eqn:6.12b}
\egee
where
\[ \mathcal{H}[w]\equiv w_t-w_{\rho\rho}+\frac{N-1}{\rho}w_\rho-p K(t)z^{p-1}w. \]
The maximum principle now implies $w\leq 0$, and hence $z_{\rho}\leq 0$ in $(0,1)\times (0,T_{\de})$. Then inequality (\ref{eqn:6.10}) follows since
\begin{eqnarray*}
\rho^Nz(\rho,t) & = & z(\rho,t)\int_0^\rho Ns^{N-1}ds\leq \int_0^\rho Nz(s,t)s^{N-1}ds \\
& \leq & \int_0^1Nz(s,t)s^{N-1}ds=\avint z=\bar{z}(t).
\end{eqnarray*}

Once $w\leq 0$ is proven, we have
\bgee
& & w_t-w_{\rho\rho}+\frac{N-1}{\rho}w_\rho=pK(t)z^{p-1}w\leq 0 \quad\mbox{in}\quad (0,1)\times (0,T_{\de}),\\
& & w\left(\frac{1}{2},t\right)\leq0,\quad w\left(1,t\right)\leq 0,\quad\mbox{for}\quad t\in (0,T_{\de}) \\
& & w(\rho,0)=\rho^{N-1}u'_{0}(\rho)\leq -c, \quad\mbox{for}\quad\frac{1}{2}<\rho<1,
\egee
which entails $w\leq -c$ in $(\frac{1}{2},1)\times (0, T_\delta)$, and finally (\ref{eqn:6.11}) holds.
\end{proof}

\begin{lemma}\label{kkl2}
Given $\varepsilon>0$ and $1<q<p$ then $\psi$ defined as
\be\label{ik1}
 \psi:=\rho^{N-1}z_\rho+\varepsilon\cdot\frac{\rho^Nz^q}{\bar{z}^{\gamma+1}},
\ee
satisfies
\begin{eqnarray}
& & \mathcal{H}[\psi]\leq -\frac{2q\varepsilon}{\bar{z}^{\gamma+1}}z^{q-1}\psi+\frac{\varepsilon \rho^Nz^q}{\bar{z}^{2(\gamma+1)}}\Big\{2q\varepsilon z^{q-1}-(\gamma+1)\bar{z}^{\gamma-r\gamma}\avint z^p  \nonumber\\
& & \quad -(p-q)z^{p-1}\bar{z}^{\gamma+1-r\gamma}\Big\} \quad \mbox{in}\quad (0,1)\times (0,T_{\de}).
 \label{eqn:6.13h}
\end{eqnarray}
 \label{lem:6.4}
\end{lemma}
The proof of Lemma \ref{kkl2} follows the same steps as the proof of inequality (28) in \cite{ln09}, which holds for $r=1,$ and thus it is omitted.

Observe that when $p>\frac{N}{N-2}$ there is $1<q<p$ such that $N>\frac{2p}{q-1}$ and thus the following quantities
\begin{equation}
A_1\equiv\sup_{0<\de<1}\frac{1}{\bar{u}_0^{\mu}} \avint u_0^p=\lambda^{\mu-p}\alpha_1, \quad A_2\equiv\inf_{0<\de<1}\frac{1}{\bar{u}_0^{\mu}} \avint u_0^p=\lambda^{\mu-p}\alpha_2
 \label{eqn:6.14}
\end{equation}
are finite due to \eqref{at1}. The following result, which is a modification of Lemma 3.3 in \cite{ln09} for $r=1,$ provides a key estimate of the $L^p-$norm of $z$ in terms of $A_1$ and $A_2$ and, since it is a core result for the proof of Theorem  \ref{thm:6.1}, we will sketch its proof shortly.

\begin{proposition}\label{lem4}
There exist $0<\delta_0<1$ and $0<t_0\leq 1$ independent of any $0<\delta\leq \delta_0,$ such that the following estimate is satisfied
\begin{equation}\label{lue}
\frac{1}{2} A_2\ol{z}^{\mu}\leq  \avint z^p\,dx\leq 2 A_1 \ol{z}^{\mu},
\end{equation}
for any $0<t<\min\{t_0,T_{\de}\}.$
\end{proposition}
The proof of the above proposition requires some auxiliary results shown below.

Take $0<t_0(\delta)<T_\delta$ to be the maximal time for which inequality (\ref{lue}) holds true in $0<t<t_0(\delta),$ then we have
\begin{equation}
\frac{1}{2}A_2\bar{z}^\mu\leq \avint z^p\leq 2A_1\bar{z}^\mu, \quad\mbox{for}\quad 0<t<t_0(\delta).
 \label{eqn:6.18}
\end{equation}
We consider the case $t_0(\delta)\leq 1,$ since otherwise there is nothing to prove.

Now the first auxiliary result states

\begin{lemma}\label{nkl2}
There exists $0<t_1<1$ such that
\begin{equation}
\bar{z}(t)\leq 2\bar{u}_0, \quad 0<t<\min\{ t_1, t_0(\delta)\},
 \label{eqn:6.18h}
\end{equation}
for any $0<\delta<1$.
 \label{lem:6.6}
\end{lemma}

\begin{proof}
Since $r\geq 1$ and $t_0(\delta)\leq 1$, it follows that
\[ \frac{d\bar{z}}{dt}\leq 2A_1e^{1+r\gamma-p}\bar{z}^{\mu-r\gamma}, \quad 0<t<t_0(\delta), \]
taking also into account relations (\ref{nnt}) and (\ref{eqn:6.8h}).

Setting $C_1=2A_1e^{1+r\gamma-p}$, we obtain
\[ \ol{z}(t)\leq \left[\bar{u}_0^{1+r\ga-\mu}-C_1(\mu-r\ga-1)t\right]^{-\frac{1}{\mu-r\ga-1}} \]
by (\ref{eqn:6.15}). Therefore, (\ref{eqn:6.18h}) holds for any $0<t<\min\{ t_1, t_0(\delta)\}$ where $t_1$ is estimated as
\bgee
t_1\leq\min\left\{\frac{1-2^{1+r\ga-\mu}}{C_1(\mu-r\ga-1)}\ol{u}_0^{1+r\ga-\mu},1\right\},
\egee
and it is independent of any $0<\delta<1.$
\end{proof}
Another fruitful estimate is provided by the next auxiliary result
\begin{lemma}\label{nkl3}
There exist $0<\delta_0<1$ and $0<\rho_0<\frac{3}{4}$ such that for any $0<\delta\leq \delta_0$  the following estimate holds
\begin{equation}
\frac{1}{\vert\Omega\vert}\int_{B(\rho_0,0)}z^p\leq \frac{A_2}{8}\bar{z}^\mu, \quad\mbox{for}\quad 0<t<\min\{ t_1, t_0(\delta)\}.
 \label{eqn:6.19}
\end{equation}
\end{lemma}
\begin{proof}
First observe that
\begin{equation}
\bar{u}_0\leq \bar{z}(t)\leq 2\bar{u}_0, \quad\mbox{for}\quad 0<t<\min\{t_1, t_0(\delta)\},
 \label{eqn:6.20}
\end{equation}
follows from (\ref{eqn:6.9}) and (\ref{eqn:6.18h}). Then, $\avint z^p$ is controlled by (\ref{lue}) for $0<\min\{ t_1, t_0(\delta)\}$. Since $p>q$ then Young's inequality guarantees that the second term of the right-hand side in (\ref{eqn:6.13h}) is negative for $0<t<\min\{ t_1, t_0(\delta)\}$, uniformly in $0<\delta<1$, provided that $0<\varepsilon\leq \varepsilon_0$ for some $0<\varepsilon_0\ll 1.$ Thus
\begin{equation}
\mathcal{H}[\psi]\leq
-\frac{2q\vep z^{q-1}}{\bar{z}^{\gamma+1}}
\psi \quad\mbox{in}\quad (0,1)\times (0,\min\{ t_1, t_0(\delta)\}).
 \label{eqn:6.21}
\end{equation}
Due to (\ref{eqn:6.10}) and (\ref{eqn:6.20}), we also have
\begin{eqnarray*}
\psi & = & \rho^{N-1}z_\rho+\varepsilon\cdot\frac{\rho^Nz^q}{\bar{z}^{\gamma+1}} \leq \rho^{N-1}z_\rho+\varepsilon\cdot\rho^{N(1-q)}\bar{z}^{q-\gamma-1} \\
& \leq & \rho^{N-1}z_\rho+C\cdot\varepsilon\rho^{N(1-q)}\quad\mbox{in}\quad (0,1)\times (0,\min\{ t_1, t_0(\delta)\})
\end{eqnarray*}
which, for $0<\varepsilon\leq \varepsilon_0,$ entails
\begin{equation}
\psi\left(\frac{3}{4},t\right)<0, \quad 0<t<\min\{ t_1, t_0(\delta)\}
 \label{eqn:6.22}
\end{equation}
by (\ref{eqn:6.11}) and provided that $0<\varepsilon_0\ll 1.$

Additionally \eqref{ik1} for $t=0$ gives
\begin{equation}
\psi(\rho,0) = \rho^{N-1}\left(\lambda \varphi_{\delta}'(\rho)+\varepsilon\lambda^{q-\gamma-1}\rho\cdot\frac{\varphi_\delta^q}{\bar{\varphi}_\delta^{\gamma+1}}\right).
 \label{eqn:6.23}
\end{equation}
Now if $0\leq \rho<\delta$ and $\varepsilon$ are chosen small enough and independent of $0<\de<\de_0,$ then the right-hand side of (\ref{eqn:6.23}) is estimated as follows:
\bgee
\rho^{N}\lambda\left( -a\delta^{-a-2}+\varepsilon\lambda^{q-\gamma-2}\cdot \frac{\varphi_\delta^q}{\bar{\varphi}_\delta^{\gamma+1}} \right)\lesssim \rho^{N}\lambda\left( -a\delta^{-a-2}+\varepsilon\lambda^{q-\gamma-2}\cdot \delta^{-aq} \right)\lesssim 0
\egee
since also
\[ \displaystyle\frac{\varphi_\delta^q}{\bar{\varphi}_\delta^{\gamma+1}}\lesssim \delta^{-aq}, \quad \delta\downarrow 0,\quad\mbox{uniformly in}\quad 0\leq \rho<\delta,\]
holds by (\ref{id}) and (\ref{ine}) for $m=1,$ taking also into account that $a+2=ap>ak.$

On the other hand, if $\delta\leq \rho\leq 1$ then we obtain
\begin{equation}
\psi(\rho,0)=\rho^N\lambda\left(-a\rho^{-a-1}+\varepsilon\lambda^{q-\gamma-1}\frac{\rho^{-aq+1}}{\bar{\varphi}_\rho^{\gamma+1}}\right),
 \label{eqn:6.24}
\end{equation}
by using again (\ref{ine}) for $m=1.$ Since $a+2=ap>aq$ implies $-a-1<-aq+1$, we  derive
\[ \psi(\rho,0)<0, \quad \de\leq \rho\leq \frac{3}{4}, \]
for any $0<\delta\leq \delta_0$ and $0<\varepsilon\leq \varepsilon_0$, provided $\varepsilon_0$ is chosen sufficiently small.

Consequently we deduce
\begin{equation}
\psi(\rho,0)<0, \quad 0\leq \rho\leq \frac{3}{4},
 \label{eqn:6.25}
\end{equation}
for any $0<\delta\leq \delta_0$ and $0<\varepsilon\leq \varepsilon_0$, provided $0<\varepsilon_0\ll 1.$

Combining (\ref{eqn:6.21}), (\ref{eqn:6.22}) and (\ref{eqn:6.25}) we end up with
\[ \psi=\rho^{N-1}z_\rho+\varepsilon\cdot\frac{\rho^Nz^q}{\bar{z}^{\gamma+1}}\leq 0 \quad \mbox{in $(0,\frac{3}{4})\times (0,\min\{ t_1, t_0(\delta)\})$},  \]
which implies
\begin{equation}
z(\rho,t)\leq \left( \frac{\varepsilon}{2}(q-1)\right)^{-\frac{1}{q-1}}\cdot \rho^{-\frac{2}{q-1}}\cdot\bar{z}^{\frac{\gamma-1}{q-1}}(t) \quad \mbox{in $(0,\frac{3}{4})\times (0,\min\{ t_1, t_0(\delta)\})$}.
 \label{eqn:6.26}
\end{equation}
Since $-\frac{2}{q-1}\cdot p+N-1>-1$ due to $N>\frac{2p}{q-1},$ we finally obtain (\ref{eqn:6.19}) for some $0<\rho_0<\frac{3}{4}$.
\end{proof}
\begin{rem}\label{nkl1}
It is worth noting that relation \eqref{eqn:6.26} implies that if $z(\rho,t)$ blows up then this can only happen in the origin $\rho=0;$ that is, only a single-point blow-up is possible. In particular if we define
\[ \mathcal{S}=\{ x_0\in \overline{\Omega} \mid \exists x_k\rightarrow x_0, \ \exists t_k\uparrow T_\de, \ \lim_{k\rightarrow\infty}z(x_k, t_k)=+\infty \}, \]
to be the blow-up set of $z$  then ${\mathcal S}=\{0\}$ in the case $z$ blows up in finite time.
\end{rem}
Next we prove the key estimate \eqref{lue} using essentially  Lemmata \ref{nkl2} and \ref{nkl3}.

\begin{proof}[Proof of Proposition \ref{lem4}]
By virtue of (\ref{eqn:6.15}) and since $\frac{p-1}{r}<\delta$, there holds that $\ell=\frac{\mu}{p}>1.$ We can easily see that $\theta=\displaystyle{\frac{z}{\ol{z}^{\ell}}}$ satisfies
\bgee
&&\theta_t=\Delta \theta+ e^{(r\ga+1-p)t}\left[\frac{z^p}{\ol{z}^{\ell}\left(\avint z^r\right)^{\ga}}-\frac{\ell z\avint z^p}{\ol{z}^{\ell+1}\left(\avint z^r\right)^{\ga}} \right], \quad\mbox{in}\quad Q_0\\
&&\frac{\partial\theta}{\partial\nu}=0\quad\mbox{on}\quad\partial\Omega\times(0, \min\{ t_0, T_\delta\}),\\
&& \theta(x,0)=\frac{z(x,0)}{\bar{z}_0^{\ell}}\quad\mbox{in}\quad \Omega.
\egee
Now due to (\ref{eqn:1.17}), (\ref{eqn:6.9}), (\ref{eqn:6.10}), (\ref{eqn:6.18}), and (\ref{eqn:6.18h}), there holds that
\[ \left\Vert \theta,\ \frac{z^p}{\ol{z}^{\ell}\left(\avint z^r\right)^{\ga}}, \ \frac{\ell z\avint z^p}{\ol{z}^{\ell+1}\left(\avint z^r\right)^{\ga}}\right\Vert_{L^\infty((\Omega\setminus B(0,\rho_0))\times \min\{ t_1, t_0(\delta)\})} \leq C \]
uniformly in $0<\delta\leq \delta_0$.

Therefore, by the standard parabolic regularity, see DeGiorgi-Nash-Moser estimates in \cite[pages 144-145]{Lie96}, there is $0<t_2\leq t_1$ independent of $0<\de\leq \de_0$ such that
\[ \sup_{0< t<\min\{t_2, t_0(\delta)\}}\left\Vert \theta^p(\cdot,t)-\theta^p(\cdot,0)\right\Vert_{L^1(\Omega\setminus B(0,\rho_0))}\leq \frac{A_2}{8}\vert\Omega\vert, \]
which implies
\begin{equation}
\left\vert \frac{1}{\vert\Omega\vert}\int_{\Omega\setminus B(0,\rho_0))}\frac{z^p}{\bar{z}^\mu}-\frac{1}{\vert\Omega\vert}\int_{\Omega\setminus B(0,\rho_0)}\frac{z_0^p}{\bar{z}_0^\mu} \right\vert \leq \frac{A_2}{8}, \quad 0<t<\min\{ t_2, t_0(\delta)\},
 \label{eqn:6.28}
\end{equation}
for any $0<\delta\leq \delta_0$. Inequalities (\ref{eqn:6.19}) and (\ref{eqn:6.28}) entail
\[ \left\vert \avint \frac{z^p}{\ol{z}^\mu}-\avint \frac{z_0^p}{\ol{z}_0^\mu}\right\vert\leq \frac{3 A_2}{8}, \quad\mbox{for}\qquad 0<t<\min\{ t_2, t_0(\delta)\}\quad\mbox{and}\quad 0<\delta\leq \delta_0, \]
and hence
\begin{equation}\label{tbsd13}
\frac{5A_2}{8}\leq \avint \frac{z^p}{\ol{z}^{\mu}} \leq \frac{11 A_1}{8}, \quad 0<t<\min\{ t_2, t_0(\de)\}, \ 0<\delta\leq \delta_0,
\end{equation}
taking into account that
\[ A_2\leq \avint\frac{z_0^p}{\bar{z}_0^\mu}\leq A_1. \]
Therefore, if we take $t_0(\delta)\leq t_2$ then it follows that
\[ \frac{1}{2}A_2\bar{z}^\mu<\frac{5}{8}A_2\bar{z}^\mu\leq \avint z^p\leq \frac{11}{8}A_1\bar{z}^\mu<2A_1\bar{z}^\mu, \quad 0<t<t_0(\delta), \]
and by a continuity argument we deduce that
\[ \frac{1}{2}A_2\bar{z}^\mu\leq \avint z^p\leq 2A_1\bar{z}^\mu,  \quad 0<t<t_0(\delta)+\eta,\]
for some $\eta>0,$ which contradicts the definition of $t_0(\delta)$.

Consequently, we obtain $t_2<t_0(\delta)$ for any $0<\delta\leq \delta_0$, and the proof is complete with $t_0=t_2$.
\end{proof}
Now we have all the ingredients to proceed to the proof of the main result of this section.
\begin{proof}[Proof of Theorem \ref{thm:6.1}]
Since $t_0\leq t_1$ in (\ref{eqn:6.18h}), we have
\begin{eqnarray}
K(t) & &\geq \frac{1}{\left(\avint z^r\right)^\gamma}\geq
\frac{1}{\left( \avint z^p\right)^{\frac{r\gamma}{p}}}\geq
\left( \frac{1}{2A_1\bar{z}^\mu}\right)^{\frac{r\gamma}{p}} = \left( \frac{1}{2A_1}\right)^{\frac{r\gamma}{p}}\cdot\left(
\frac{1}{\bar{z}}\right)^{\frac{r\gamma}{p}\mu} \nonumber\\
& & = \left( \frac{1}{2\alpha_1}\right)^{\frac{r\gamma}{p}}\cdot\left( \frac{1}{2\bar{\varphi}_\delta}\right)^{\frac{r\gamma}{p}\mu}\lambda^{-r\gamma}\geq d\lambda^{-r\gamma}\equiv D, \quad 0<t<\min\{ t_0, T_\delta\},
 \label{eqn:6.33}
\end{eqnarray}
by virtue of (\ref{eqn:6.7}) and (\ref{eqn:6.14}). Since $0<\lambda\leq \lambda_0(d)$, then inequality (\ref{eqn:6.4}) applies to derive
\begin{equation}
\Delta u_0+Du_0^p\geq 2u_0^p
 \label{eqn:6.34}
\end{equation}
for any $0<\delta\leq \delta_0$.

By using (\ref{eqn:6.33}) and (\ref{eqn:6.34}) the comparison principle yields that the solution $z$ of \eqref{tbsd3}-\eqref{tbsd3b} satisfies
\begin{equation}
z\geq \tilde{z} \quad\mbox{in}\quad Q_0\equiv\Omega\times (0, \min\{ t_0, T_\delta\}),
 \label{eqn:6.35}
\end{equation}
where $\tilde z=\tilde z(x,t)$ solves the following
\bge
&&\tilde{z}_t=\Delta \tilde{z}+D\tilde{z}^p, \quad\mbox{in}\quad Q_0,\label{eqn:6.36}\\
&&\frac{\partial \tilde{z}}{\partial\nu}=0,\quad\mbox{on}\quad\partial\Omega\times(0, \min\{ t_0, T_\delta\}),\\
&&\tilde{z}(|x|,t)=u_0(\vert x\vert)\quad\mbox{in}\quad \Om.
 \label{eqn:6.36}
\ege
Let us now introduce
\[ h(x,t):=\tilde{z}_t(x,t)-\tilde{z}^p(x,t), \]
then due to (\ref{eqn:6.34}) and (\ref{eqn:6.36})  $h$ satisfies
\begin{eqnarray*}
h_t & = & \Delta h+p(p-1) \tilde{z}^{p-2} |\nabla \tilde{z}|^2+ D p \tilde{z}^{p-1}\,h \\
& \geq & \Delta h+ D p \tilde{z}^{p-1}\,h \quad\mbox{in}\quad Q_0,
\end{eqnarray*}
and
\bgee
h(x,0)= \Delta \tilde{z}(x,0)+D\tilde{z}^p(x,0)-\tilde{z}^p(x,0)=\Delta u_0+(D-1)u_0^p\geq u_0^p>0,\quad\mbox{in}\quad \Om
\egee
with boundary condition
\bgee
\frac{\partial h}{\partial \nu}=0\quad\mbox{on}\quad\partial\Omega\times(0, \min\{ t_0, T_\delta\}).
\egee
Then the maximum principle entails that $h>0$ in $Q_0$, that is,
\begin{equation}
\tilde{z}_t>\tilde{z}^p\quad\mbox{in}\quad Q_0.
 \label{eqn:6.37}
\end{equation}
Inequality (\ref{eqn:6.37}) implies
\[ \tilde{z}(0,t)\geq\left(\frac{1}{z_0^{p-1}(0)}-(p-1)t\right)^{-\frac{1}{p-1}}=\left\{\left(\frac{\de^{a}}{\la(1+\frac{a}{2})}\right)^{p-1}-(p-1)t\right\}^{-\frac{1}{p-1}}
\]
for $0<t<\min\{ t_0, T_\delta\}$, and therefore,
\begin{equation}
\min\{ t_0, T_\delta\}<\frac{1}{p-1}\cdot
\left( \frac{\delta^a}{\lambda(1+\frac{a}{2})}\right)^{p-1}.
 \label{eqn:6.38}
\end{equation}
For $0<\delta\ll 1$, the right-hand side on (\ref{eqn:6.38}) is less than $t_0$, and then $T_\delta<+\infty$ follows. Furthermore, by \eqref{eqn:6.38} $T_\de \to 0$ as $\de\to 0$ and the proof is complete.
\end{proof}
\begin{rem}
The blowing up solution $u$ obtained in Theorem \ref{thm:6.1} exhibits a single-point blow-up at the origin $\rho=0.$ Recalling that $z=e^tu$ we obtain the occurrence of single-point blow-up for $u$ in view of Remark \ref{nkl1}.

An alternative way to prove single-point blow-up is by virtue of the following estimate
\begin{equation}
\avint z^p\,dx=\frac{1}{|B_1(0)|}\int_0^1 \rho^{N-1} z^p\,d\rho \leq C,\quad\mbox{for}\quad 0<t\leq T_{\delta},
 \label{eqn:7.1}
\end{equation}
which holds due to \eqref{lue} and \eqref{eqn:6.18h}, taking $0<\delta\ll 1$ small enough such that $T_\delta\leq t_0.$
Then since $z=z(\rho, t)$ is radially decreasing, then (\ref{eqn:7.1}) implies that ${\mathcal S}=\{0\}.$
\end{rem}
\section{Blow-up rate and blow-up pattern}\label{sec:7}
One of our purposes in the current section is to determine the blow-up rate of the diffusion-driven blowing up solution provided by Theorem \ref{thm:6.1}. We also intend to identify its blow-up pattern (profile) and thus reveal the formed patterns anticipated in this {\it DDI} event.
\begin{theorem}\label{tbu}
Let $N\ge 3,\;\max\{r, \frac{N}{N-2}\}<p<\frac{N+2}{N-2}$ and $\frac{2}{N}<\frac{p-1}{r}<\gamma.$  Then the blow-up rate of the diffusion-induced blowing-up solution of Theorem \ref{thm:6.1} is determined as follows
\be
\Vert u(\cdot, t)\Vert_\infty \ \approx \ (T_{\max}-t)^{-\frac{1}{p-1}}, \quad t\uparrow T_\de,\label{ik2}
\ee
 where $T_{\max}$ stands for the blow-up time.
\end{theorem}
\begin{proof}
We first note that
\bge\label{nkl4}
0<K(t)=\frac{e^{(1+r\gamma-p)t}}{\Big(\displaystyle\avint z^r \Big)^{\gamma}}\leq C<\infty,
\ege
by virtue of \eqref{eqn:7.1} and in view of H\"{o}lder's inequality since $p>r.$

Consider now $\Phi$ satisfying
\bgee
&&\Phi_t=\Delta \Phi+C\Phi^{p},\quad\mbox{in}\quad \Om\times (0,T_{max}),\\
&&\frac{\partial \Phi}{\partial \nu}=0,\quad\mbox{on}\quad \partial\Om\times(0,T_{max}), \\
&&\Phi(x,0)=z_0(x), \quad\mbox{in}\quad \Om,
\egee
then via comparison $z\leq \Phi$ in $\Om\times (0,T_{max}).$

Yet it is known, see \cite[Theorem 44.6]{qs07}, that
\bgee
|\Phi(x,t)|\leq C_{\eta}|x|^{-\frac{2}{p-1}-\eta}\quad\mbox{for}\quad \eta>0,
\egee
when $x\in \Om,\;0<t<T_{max},$ and thus
\bge\label{tbsd18}
|z(x,t)|\leq C_{\eta}|x|^{-\frac{2}{p-1}-\eta}\quad\mbox{for}\quad (x,t)\in \Om\times (0, T_{max}),
\ege
which by virtue of \eqref{nkl4}, \eqref{tbsd18} and using also standard parabolic estimates entails that
\bge\label{nkl5}
z\in \mathcal{BUC^{\sigma}}\left(\left\{\rho_0<|x|<1-\rho_0\right\}\times \left(\frac{T_{max}}{2}, T_{max}\right)\right)
\ege
for some $\sigma\in(0,1)$ and each $0<\rho_0<1,$ where $\mathcal{BUC^{\sigma}}(M)$ denotes in general the Banach space of all bounded and uniform $\sigma-$H\"{o}lder continuous functions $h:M\subset\R^N\to \R;$ see also \cite{qs07}.

Consequently \eqref{nkl5} implies that  $\lim_{t\to T_{max}}z(x,t)$ exists and it is finite for all $x\in B_1(0)\setminus\{0\}.$

Recalling that $\displaystyle{\frac{2p}{p-1}}<N$ (or equivalently  $p>\displaystyle{\frac{N}{N-2}},\; N>2$ ) then by using \eqref{nkl4},\eqref{tbsd18} and in view of the dominated convergence theorem we derive
\bge\label{tbsd18a}
\lim_{t\to T_{max}} K(t)=\omega\in(0,+\infty).
\ege
Applying now Theorem 44.3(ii) in \cite{qs07}, taking also into account \eqref{tbsd18a}, we can find a constant $C_{u}>0$ such that
\bge\label{ube}
\left|\left|z(\cdot,t)\right|\right|_{\infty}\leq C_{u}\left(T_{max}-t\right)^{-\frac{1}{(p-1)}}\quad\mbox{in}\quad (0, T_{max}).
\ege
On the other hand, setting $N(t):=\left|\left|z(\cdot,t)\right|\right|_{\infty}=z(0,t)$ then $N(t)$ is differentiable for almost every $t\in(0,T_{\de}),$ in view of  \cite{fmc85}, and it also satisfies
\bgee
\frac{dN}{dt}\leq K(t) N^p(t).
\egee
Now since $K(t)\in C([0,T_{\max}))$ is bounded in any time interval $[0,t],\; t<T_{\max},$ then upon integration we obtain
\bge\label{lbe}
\left|\left|z(\cdot,t)\right|\right|_{\infty}\geq C_l\left(T_{\de}-t\right)^{-\frac{1}{(p-1)}}\quad\mbox{in}\quad (0, T_{\max}),
\ege
for some positive constant $C_l.$

Since $z(x,t)=e^t u(x,t) $ then by virtue of \eqref{ube} and \eqref{lbe} we obtain
\bgee
\widetilde{C}_l\left(T_{max}-t\right)^{-\frac{1}{(p-1)}}\leq\left|\left|u(\cdot,t)\right|\right|_{\infty}\leq \widetilde{C}_u\left(T_{max}-t\right)^{-\frac{1}{(p-1)}}\quad\mbox{for}\quad t\in(0, T_{max}),
\egee
where now $\widetilde{C}_l, \widetilde{C}_u$ depend on $T_{max},$
which actually leads to \eqref{ik2}.
\end{proof}
\begin{rem}\label{nkl6}
Condition \eqref{ik2} implies that the diffusion-induced blow-up of Theorem \ref{thm:6.1} is of type I, i.e. the blow-up mechanism is controlled by the ODE part of \eqref{bsd1}.

In contrast, for the finite-time blow-up furnished by Proposition \ref{lem:1.2} and Theorems \ref{thm:1.2} and \ref{obu1} we cannot derive a blow-up as in \eqref{ik2} since the blow-up of some $L^{\ell}-$norm, $\ell\geq 1,$ in each of these cases entails that \bgee
K(t)=\frac{e^{(1-p)t}}{\Big(\displaystyle\avint u^r \Big)^{\gamma}}\to 0\quad\mbox{as}\quad t\to T_{max},
\egee
and thus the approach of Theorem \ref{tbu} fails. This might be an indication that in the preceding cases finite-time blow-up is rather of type II.
\end{rem}
\begin{rem}\label{nkl7}
First observe that \eqref{tbsd18} provides a rough form of the blow-up pattern for $z$ and thus for $u$ as well. Nonetheless, due to \eqref{nkl4} then the non-local problem \eqref{tbsd3}-\eqref{tbsd3b} can be treated as the corresponding local one for which the following more accurate asymptotic blow-up profile, \cite{mz98}, is available
\bgee
\lim_{t\to T_{max}}z(|x|,t)\sim C\left[\frac{|\log |x||}{|x|^2}\right]\quad\mbox{for}\quad |x|\ll 1.
\egee
Therefore using again that $z=e^tu$ we derive a similar asymptotic blow-up profile for the driven-induced blowing up solution $u.$ This actually reveals the form of the developed patterns which are induced as a result of the {\it DDI} and will be numerically verified  in a forthcoming paper.
\end{rem}
\section{Conclusions}\label{conc}
The main purpose of the current manuscript is to unveil under which circumstances the dynamics of the interaction of the two morphogens (activator and inhibitor), described by the Gierer-Meinhardt system \eqref{gm1}-\eqref{gm4}, can be controlled by governing only the dynamics of the activator itself given by the non-local problem \eqref{bsd1}-\eqref{bsd3}. We derive some global-in-time existence as well as blow-up results both in finite and in infinite time for \eqref{bsd1}-\eqref{bsd3}. Global-in-time existence results guarantee the controlled growth of the activator as described by \eqref{bsd1}-\eqref{bsd3} whereas finite-time and infinite-time blow-up results are relevant with the activator's unlimited growth. We discovered, that there are cases, see Proposition \ref{ats1} and Theorems \ref{thm:1.2} and \ref{thm:1.2a}, where there is a serious discrepancy between the dynamics of the full system \eqref{gm1}-\eqref{gm4} and those of the non-local problem \eqref{bsd1}-\eqref{bsd3}. On the other hand, under other circumstances, see  Theorems \ref{thm:1.3}-\ref{thm:1.7}, then both  \eqref{gm1}-\eqref{gm4} and (\ref{bsd1})-(\ref{bsd3}) ensemble the same long-time dynamics.  In particular, in Theorems \ref{obu1} and \ref{thm:1.6} we show that the occurrence of some invariant regions for an associated dynamical system is vital in order to control the dynamics of the non-local problem \eqref{bsd1}-\eqref{bsd3} and thus the activator's growth. In addition,  we prove that under a Turing condition a {\it DDI} occurs, which is exhibited in the form of a driven-diffusion blow-up. The resulting destabilization enables the formation of some patterns, as anticipated in a Turing instability case. The form of the observed patterns is completely described via the study of the blow-up profile, see Remark \ref{nkl7}. Consequently, in that case the pattern formation for activator's concentration can be efficiently predicted and prescribed by the dynamics of the non-local problem \eqref{bsd1}-\eqref{bsd3}.
\section*{Acknowledgments}
This work was supported by JSPS Grant-in-Aid Scientific Research (A) 26247013 and Core-to-Core project. Part of the current work was inspired and initiated when the first author was visiting the Department of System Innovation of  Osaka University. He would like to express his gratitude for the warm hospitality.

The authors would also like to thank the anonymous reviewers for the their stimulating comments, which  substantially improved the form of the manuscript.

\end{document}